\documentclass[reqno,11pts]{amsart}
\usepackage{amssymb,amsmath,amsthm,amstext,amsfonts}
\usepackage[dvips]{graphicx}
\usepackage{psfrag}
\theoremstyle{plain}
\newtheorem{maintheorem}{Theorem}

\newtheorem{maincorollary}{Corollary}

\newtheorem{theorem}{Theorem }[section]
\newtheorem{proposition}[theorem]{Proposition}
\newtheorem{lemma}[theorem]{Lemma}
\newtheorem{corollary}[theorem]{Corollary}

\theoremstyle{definition} \theoremstyle{remark}
\newtheorem{remark}[theorem]{Remark}

\newtheorem{definition}[theorem]{Definition}

\newcommand{\diam}{\operatorname{diam}}

\newcommand{\topp}{\operatorname{top}}
\newcommand{\Ptop}{P_{\topp}}

\newcommand{\al} {\alpha}       
        
\newcommand{\ga} {\gamma}    
       \newcommand{\De}{\Delta}

\newcommand{\vep}{\varepsilon}

\newcommand{\la} {\lambda}

\newcommand{\si} {\sigma}

\newcommand{\cL}{\mathcal{L}}

\newcommand{\cB}{\mathcal{B}}
\newcommand{\cO}{\mathcal{O}}

\newcommand{\cN}{\mathcal{N}}
\newcommand{\cQ}{\mathcal{Q}}

\newcommand{\var}{\text{var}_{\theta}}

\newcommand{\osc}{\operatornamewithlimits{osc}}

\newcommand{\cQn}{\cQ^{(n)}}
\newcommand{\cQk}{\cQ^{(k)}}
\newcommand{\ov}{\overline}

\begin{document}

\title{Correlation decay and recurrence asymptotics for some robust
        nonuniformly hyperbolic maps}
\author{Paulo Varandas}

\address{Paulo Varandas, Instituto de Matem\'atica,
Universidade Federal do Rio de Janeiro C. P. 68.530, 21.945-970, Rio
de Janeiro, RJ-Brazil}  \email{varandas@impa.br}

\date{\today}



\maketitle

\begin{abstract}
We study decay of correlations, the asymptotic distribution of
hitting times and fluctuations of the return times for a robust
class of multidimensional non-uniformly hyperbolic transformations.
Oliveira and Viana~\cite{OV07} proved that there is a unique
equilibrium state $\mu$ for a large class of non-uniformly expanding
transformations and H\"older continuous potentials with small
variation. For an open class of potentials with small variation, we
prove quasi-compactness of the Ruelle-Perron-Frobenius operator in a
space $V_\theta$ of functions with essential bounded variation that
strictly contain H\"older continuous observables. We deduce that the
equilibrium states have exponential decay of correlations.
Furthermore, we prove exponential asymptotic distribution of hitting
times and log-normal fluctuations of the return times around the
average $h_\mu(f)$.
\end{abstract}


\section{Introduction}

Given a measure preserving discrete dynamical system it follows by
earlier work of Poincar\'e that the orbit of almost every point will
return to any arbitrary small neighborhood of it. Return times are
strongly related to the complexity of the dynamical system and, in
many cases, the entropy coincides with the exponential growth rate
of the return times to decreasing sequences of nested sets.
A related concept is the one of hitting times. Typical orbits will
eventually visit any positive measure set with some frequency which
coincides, in average, with the measure of the set. In particular
return times reflect both hitting times statistics and the rate at
which the system is mixing. Some questions that naturally arise are:

\vspace{.1cm} $1$. What is the distribution of the hitting times of
the system when one considers sets of arbitrary small measure?
\vspace{.1cm}

$2$. How do return times oscillate around their average?
\vspace{.1cm}

In general, an answer to these questions involves a deep knowledge
of the system's chaotic features (expressed in terms of mixing
properties) combined with information on the measure of sets in
small scales. For the later it is usually enough the invariant
measure to satisfy some local equilibrium property.

The study of hitting and first return times and their connection
with hyperbolicity, speed of mixing and dimension theory achieved
many recent developments and became an important ingredient to
characterize the statistical properties of dynamical systems. Some
of the first attempts to understand possible phenomena in a
hyperbolic setting include works by Pitskel~\cite{Pit91}, for
transitive Markov chains, by Hirata~\cite{Hi93} for Axiom A
diffeomorphisms and by Collet~\cite{Col96} in the context of
expanding maps of the interval with a spectral gap, where they
proved that the distribution of hitting and return times is
asymptotically exponential. The extension of these results beyond
the scope of uniform hyperbolicity is as a challenge and gained
special attention in the few past years following the recent
interest and developments on the thermodynamical formalism for
nonuniformly hyperbolic transformations. Nevertheless, and despite
the effort of many authors, a general picture is still far from
complete. Some recent contributions among many others include works
by Collet, Galves~\cite{CG93} on maps of the interval with
indifferent fixed points; Galves, Schmitt \cite{GaSc97} for systems
satisfying a $\varphi$-mixing condition; Haydn~\cite{Hay99a,Hay99b}
for Julia sets of rational maps; Hirata, Saussol,
Vaienti~\cite{HSV99} on Poissonian laws for multiple return time
statistics for some non hyperbolic maps of the interval;
Paccaut~\cite{Pac00} on weighted piecewise expanding maps of the
interval; Abadi~\cite{Aba01} for $\alpha$-mixing stationary
processes; Collet~\cite{Col01} on Poisson laws for non-uniformly
hyperbolic maps that admit a Young tower; Bruin, Saussol,
Troubetzkoy, Vaienti \cite{BSTV03} on the return time statistics via
inducing; Saussol, Troubetzkoy, Vaienti~\cite{STV03} on the relation
between recurrence, dimension and Lyapunov exponents; and Bruin,
Todd~\cite{BT07b} on interval maps with positive Lyapunov exponent,
just to mention some of the most recent advances.
In particular, there are several evidences of an intricate relation
between the asymptotic behavior of the distribution of hitting times
and the memory loss of the system expressed in terms of good mixing
properties.
Finally, return time statistics are also useful to study the
fluctuations of the return times in the Ornstein-Weiss formula for
the metric entropy (see \cite{OW93}). In fact, it follows from the
work of Saussol~\cite{Sau01} that there is a strong connection
between fluctuations of the return times and the fluctuations in the
Shannon-MacMillan-Breiman's theorem, which are easier to study if
the invariant measure has a Gibbs property.

Here we deal with a robust class of multidimensional non-uniformly
hyperbolic transformations introduced by Oliveira and Viana in
\cite{OV07}, that contain maps obtained as deformations by isotopy
from expanding transformations as the ones considered in
\cite[Appendix]{ABV00}. Despite the existence of a (possibly
non-generating) Markov partition many difficulties arise from the
multidimensional character of the system and the absence of bounded
distortion. Oliveira and Viana developed a thermodynamical formalism
to show that there is a unique equilibrium state for every H\"older
potential with small variation and, moreover, that is satisfies a
weak Gibbs property. Our starting point to study the asymptotics of
hitting and return times as statistical properties of the
equilibrium states is to estimate the decay of correlations. That
is, the velocity at which
$$
C_n(\Phi,\Psi)
    =\Big| \int \Psi (\Phi \circ T^n) \, d\mu - \int \Phi  \,d\mu  \int \Psi \,d\mu\Big|
$$
tends to zero as $n\to \infty$ for any observable $\Psi$ and $\Psi$
in some reasonable space $V_\theta$ of functions with essential
bounded variation, that contain H\"older continuous observables and
characteristic functions at cylinders of the partitions generated
dynamically by the dynamics. For a related of potentials, we show
that the Ruelle-Perron-Frobenius operator has a spectral gap in
$V_\theta$ and deduce exponential decay of correlations and the
central limit theorem. This mixing property is enough to obtain
exponential return time statistics. Finally, a weak Gibbs property
for the equilibrium states allow us to relate return time statistics
with fluctuations of the measure of elements of dynamical partitions
and return times around the metric entropy given by
Shannon-MacMillan-Breiman and Ornstein-Weiss formulas and to obtain
log-normal fluctuations of the return times. Let us point out that
exponential return time statistics and log-normal fluctuations of
return times are robust in this nonuniformly hyperbolic setting.
Our approach to obtain exponential return time statistics, although
similar in flavor with \cite{Pac00} and \cite{GaSc97}, faces
distinct difficulties that arise from the non hyperbolicity of the
system. While the cornerstone in \cite{Pac00} was the non-Markov
property of partitions defining piecewise expanding maps of the
interval, our main difficulties lie in the lack of bounded
distortion and that the diameter of cylinders in the Markov
partition may not decrease to zero. For the sake of completeness,
let us mention that Arbieto and Matheus~\cite{AM} proved exponential
decay of correlations the equilibrium states constructed in
\cite{OV07} but considered different classes of potentials and
observables.

A very interesting question is to obtain exponential return time
statistics to balls instead of cylinders. Although this is possible
to obtain for several cases in the one-dimensional setting, the
study of return time statistics to balls presents itself as a major
difficulty in this higher dimension setting. If such a result could
be obtained it is most likely that our results can be extended to
the more general context of \cite{VV1}, where \cite{OV07} is
generalized and no Markov partition is assumed to exist.

This article is organized as follows. In Section~\ref{sec.statement}
we state our main results. In Section~\ref{sec.preliminaries} we
introduce some notations and tools that will be used in the
remaining of the paper. We prove the quasi-compactness of the
Ruelle-Perron-Frobenius operator in $V_\theta$ in
Section~\ref{sec.spectral.gap}. Finally, in
Sections~\ref{sec.exponentialdistribution} and
~\ref{sec.fluctuations} we deal with the asymptotic distribution of
hitting times and fluctuations of the return times.

\section{Statement of the Results}\label{sec.statement}

\subsection{Setting}

Let $M$ denote a compact Riemannian manifold. We say that a set $E
\subset M$ has \emph{finite inner diameter} if there exists $L>0$
such that any two points in $E$ may be joined by a curve of length
less than $L$ contained in $E$. Throughout, $f: M \to M$ will denote
a $C^1$ local diffeomorphism satisfying conditions (H1) and (H2)
below:

\begin{itemize}
\item[(H1)] There are $p \ge 1$, $q \ge 0$, and a family
$\cQ=\{Q_1\,, \dots, Q_q\,, Q_{q+1}, \dots, Q_{q+p}\}$ of pairwise
disjoint open sets whose closures have finite inner diameter and
cover the whole $M$, such that
    \begin{itemize}
    \item every $f|\bar Q_i$ is a homeomorphism onto its image
    \item if $f(Q_i) \cap Q_j \neq \emptyset$ then $f(Q_i) \supset
            Q_j$ and, hence, $f(\bar{Q}_i) \supset \bar{Q}_j$
    \item there is $N\ge 1$ such that $f^N(Q_i)=M$ for every $i$.
    \end{itemize}
\item[(H2)] There are positive constants $\sigma>1$ and $L>0$ such that
\begin{itemize}
     \item $\|Df(x)^{-1}\|\leq \sigma^{-1}$ for every
            $x\in {Q}_{q+1} \cup \dots \cup {Q}_{q+p}$
     \item $\|Df(x)^{-1}\| \leq L$ for every
       $x\in {Q}_1 \cup \cdots \cup {Q}_q$,
\end{itemize}
where $L$ is assumed to be close to be $1$ in order to satisfy the
relations \eqref{eq.relation.L}.
\end{itemize}

These conditions, which roughly mean that the transformation is
expanding in some (topologically) large region of $M$ but may admit
contracting behavior in the complement, are satisfied a large class
of local diffeomorphisms obtained by a local bifurcation of an
expanding transformation.

We will denote by $\phi: M \to \mathbb R$ an $\al$-H\"older
continuous potential with small oscillation, in the sense that it
satisfies
\begin{itemize}
\item[(H3a)] $\sup\phi-\inf\phi < \log \deg(f) - \log q$.
\end{itemize}
and condition (H3b) stated at the beginning of
Subsection~\ref{subsec.spectral.gap}. These conditions are clearly
satisfied by an open class of potentials containing the constant
ones. Note that (H1), (H2) and (H3a) are the assumptions in
\cite{OV07}.

\subsection{Equilibrium states and conformal measures}

Given a continuous transformation $f:M\to M$ and a continuous
potential $\phi:M \to \mathbb R$, an invariant probability measure
$\mu$ is an \emph{equilibrium state} for $f$ with respect to $\phi$
if it attains the supremum
$$
\Ptop(f,\phi)
    = \sup\Big\{h_\eta(f)+\int \phi\,d\eta :  \eta \; \text{is $f$-invariant}\Big\}
$$
given by the variational principle for the pressure (see e.g.
\cite{Wa82}). The \emph{Ruelle-Perron-Frobenius operator} $\cL_\phi$
is the linear operator that acts in the space $C(M)$ of continuous
functions by
$$
\cL_\phi g (x)= \sum_{f(y)=x} e^{\phi(y)} g(y).
$$
The action of the dual operator $\cL_\phi^*$ on the space $\mathcal
M (M)$ of probability measures is given by $\int g \,d\cL^*_\phi
\nu= \int \cL_\phi g\, d\nu$ for every $g \in C(M)$. We say that a
measure $\nu$ is \emph{conformal} if there exists a strictly
positive function $J_\nu f$ (Jacobian of $\nu$ with respect to $f$)
such that
 $
 \nu(f(A))= \int J_\nu f \, d\nu
 $
for every measurable set $A$ such that $f\mid A$ is injective. It is
nor difficult to see that any eigenmeasure $\nu$ for $\cL_\phi^*$
associated to a positive eigenvalue $\lambda$ is a conformal measure
for $f$ and that $J_\nu f=\lambda e^{-\phi}$.

A sequence of positive integers $(n_k)_{k\geq1}$ is
\emph{non-lacunary} if it is increasing and $n_{k+1}/n_k\to 1$ when
$k$ tends to infinity. Consider the partition
$\cQ^{(n)}=\bigvee_{j=0}^{n-1} f^{-j}\cQ$ and let $Q_n(x)$ be the
element of $\cQn$ that contains $x$.

\begin{definition} A probability measure $\nu$ is a
\emph{non-lacunary Gibbs measure} if there is $K>0$ so that, for
$\nu$-almost every $x \in M$ there exists some non-lacunary sequence
$(n_k)_{k\geq 1}$, depending on $x$, such that
\begin{equation*}\label{eq. Gibbs at hyperbolic times}
K^{-1} \leq \frac{\nu(Q_{n_k}(x))}
                 {\exp(-P\,n_k + S_{n_k}\phi(y))}\leq K
\end{equation*}
for every $y\in Q_{n_k}(x)$ and every $k\ge 1$.
\end{definition}

Finally, we recall the notion of \emph{hyperbolic time} introduced
in \cite{ABV00}. We say that $n$ is a \emph{$c$-hyperbolic time} for
$x \in M$ if
\begin{equation}\label{eq. c-hyperbolic times}
\prod_{j=n-k}^{n-1} \| Df(f^j(x))^{-1}\| < e^{-ck}
 \quad \text{for every} \; 1\leq k\leq n.
\end{equation}
Since the constant $c$ will be fixed below, according to
\eqref{eq.relation.expansion}, we will refer to these simply as
hyperbolic times. We say that $Q_n \in \cQn$ is an \emph{hyperbolic
cylinder} if $n$ is a hyperbolic time for {\bf every} point in
$Q_n$. We denote by $\cQn_h$ the set of hyperbolic cylinders of
order $n$ and by $H$ the set of points that belong to the closure of
infinitely many hyperbolic cylinders. We say that a probability
measure $\nu$ is \emph{expanding} if it satisfies $\nu(H)=1$. The
next theorem summarizes the results by Oliveira, Viana ~\cite{OV07}
in this non-uniformly hyperbolic setting:

\begin{theorem}\label{t.OV2}~\cite{OV07}
Assume that $f$ is a $C^1$ local diffeomorphism such that (H1) and
(H2) hold and $\phi:M \to \mathbb R$ is an H\"older continuous
potential that satisfies (H3a). Then there exists an expanding
conformal measure $\nu$ such that $\cL^*_\phi \nu=\la \nu$, where
$\la$ denotes the spectral radius of the operator $\cL_\phi$ in the
space $C(M)$. Moreover, there is a unique equilibrium state $\mu$
for $f$ with respect to $\phi$, it is absolutely continuous with
respect to $\nu$ and it is a non-lacunary Gibbs measure.
\end{theorem}

Throughout, $\mu$ and $\nu$ will always denote the probability
measures given above.

\subsection{Statement of the main results}

First we introduce some necessary concepts. We consider a one
parameter functional space $V_\theta$, introduced in \cite{Paccaut},
using the reference partition $\cQ$ and the conformal measure $\nu$.
Given $\theta>0$ and $g \in L^\infty(\nu)$ define the
$\theta$-variation of $g$ (with respect to $\cQ$ and $\nu$) by
$$
\var(g)
    = \sum_{n \geq 1} \;\theta^n \sum_{Q_n \in \cQ^{(n)}}
    e^{S_n \phi(Q_n)} \ov\osc(g, \cQ_n),
$$
where $S_n\phi(Q_n)=\sup\{\sum_{j=0}^{n-1} \phi(f^j(x)): x \in
Q_n\}$ and $\ov\osc(g, \cQ_n)$ is the $\nu$-essential variation of
$g$ in $Q_n$ defined in Subsection~\ref{subsec.essential}. Let
$V_\theta$ be the space of functions with essential $\theta$-bounded
variation:
$$
V_\theta = \{ g \in L^\infty(\nu) : \|g\|_\theta<\infty\},
$$
where $\|\cdot\|_\theta=\|\cdot\|_\infty+\var(\cdot)$.
Paccaut~\cite{Paccaut} proved that $V_\theta$ is a Banach space. We
will say that $\mu$ satisfies \emph{exponential decay of
correlations} if there exist $C>0$ and $\xi \in (0,1)$ such that
$$
C_n(\Phi,\Psi)
 :=
\Big|
    \int \Phi (\Psi \circ f^n) d\mu - \int \Phi d\mu \int \Psi d\mu
\Big|
 \leq C \xi^n \|\Phi\|_\theta \|\Psi\|_{L^1(\nu)}.
$$
for every $\Psi \in L^1(\nu)$ and every $\Phi \in V_\theta$. A
linear operator $L$ in a Banach space $B$ is \emph{quasi-compact} if
there exists an $L$-invariant decomposition $B=B_0\oplus B_1$ of the
Banach space such that $B_0$ is finite dimensional and the spectrum
of $L|_{B_0}$ is a finite number of eigenvalues of absolute value
$r(L)$, and $r(L|_{B_1})<r(L)$. If $\dim(B_0)=1$ then we say that
$L$ has a \emph{spectral gap}. Our first main result is the
following:

\begin{maintheorem}\label{thm.spectralgap}
There exists a positive $\theta$ such that $\cL_\phi$ is
quasi-compact and has a spectral gap in the space $V_\theta$.
Moreover, $\mu$ has exponential decay of correlations in $V_\theta$
and the density $d\mu/d\nu$ belongs to $V_\theta$.
\end{maintheorem}

Fix $\theta$ as above. The previous result implies that the
correlation functions are summable. Consequently, the
\emph{asymptotic variance} $\si^2(\Phi)$ defined by
$$
 \si^2(\Phi)=\|\Phi\|_{L^1(\mu)}+2\sum_{j =1}^\infty \int \Phi \, (\Phi \circ f^j)\,d\mu
$$
is well defined for every $\Phi\in V_\theta$. Standard computations
involving the spectral gap property in the previous theorem (see e.g
\cite{You98}) are enough to obtain the Central Limit Theorem:

\begin{maincorollary}\label{c.CLT}
Assume that $\Phi \in V_\theta$ and that the asymptotic variance
$\si^2(\Phi)$ is nonzero. Then the distribution of the random
variables
$$
\frac{1}{\si(\Phi) \sqrt{n}}
    \sum_{j=0}^{n-1} \Big( \Phi\circ f^j-\int \Phi \,d\mu \Big)
$$
converges to the normal distribution $\cN(0,1)$ as $n$ tends to
infinity. Moreover, $\si(\Phi) = 0$ if and only if there exists a
measurable function $\tilde \Phi$ such that $\Phi=\tilde \Phi \circ
f -\tilde \Phi$.
\end{maincorollary}

Our next aim is to study the asymptotics of hitting times. Given a
set $Q$ consider the \emph{hitting time} $\tau_{Q}$ defined as
$$
\tau_{Q}(x)=\inf\{k \geq 1 : f^k(x) \in Q\}.
$$
We study the deviation of the hitting times from its average given
by Kac's formula: $\int \tau_Q \,d\mu = \mu(Q)^{-1}$. Indeed, we use
that the potential $\phi$ belongs to $V_\theta$ (see
Lemma~\ref{l.Vtheta.contains.Holder}) and the good mixing properties
to show exponential return time statistics for the hitting time of
most cylinders.

\begin{maintheorem}\label{thm.first.entrance.distribution}
 There are positive constants $K$ and
$\beta$, and for every $\vep>0$ there exists $N_\vep \geq 1$ such
that the following holds: for any $n \geq N_\vep$ there exists a
subset $\cQn_\vep$ of $n$-cylinders satisfying
\begin{enumerate}
\item $\mu(\cup\{Q_n: Q_n \in \cQn_\vep\}) \geq 1- \vep$;
\item For every $Q \in \cQn_\vep$
    $$ \sup_{t\geq 0}
    \Big| \mu\Big( \tau_{Q} > \frac{t}{\mu(Q)} \Big) - e^{-t} \Big|
    \leq K e^{-\beta n}.
$$
\end{enumerate}
\end{maintheorem}

This theorem asserts that the distribution of the hitting times is
asymptotically exponential and that the convergence is in a strong
sense for the majority of the cylinders. This is very useful to
study the fluctuations of the return times around the average in
Ornstein-Weiss's theorem. Since $\mu$ is ergodic, if $n \geq 1$ and
$$
R_n(x)= \inf\{k\geq 1: f^k(x) \in \cQn(x)\}
$$
denotes the \emph{$n$th return time map} then Ornstein-Weiss's
theorem asserts that
$$
h_\mu(f,\cQ)=\lim_{n\to \infty} \frac 1n \log R_n(x),
    \quad \text{for $\mu$-a.e. $x$.}
$$
We shall see later on that the diameter of $Q_n(x)$ tends to zero as
$n\to\infty$ at $\mu$-almost every $x$, which shows that $\cQ$ is a
generating partition for $\mu$. We study the fluctuation of the
random variables $\log R_n$ around the average $n h_\mu(f)$.

\begin{maintheorem}\label{thm.recurrence.fluctuations}
Assume that $\si^2(\phi)$ is positive. Then the following
convergence in distribution holds:
$$
\frac{\log R_n -n h_\mu(f)}{ \si(\phi) \sqrt n} \xrightarrow[n \to
+\infty]{\mathcal D} \cN(0,1),
$$
where $\cN(0,1)$ denotes the standard zero mean Gaussian.
\end{maintheorem}

\section{Preliminaries}\label{sec.preliminaries}

In this section we recall some necessary concepts that will be used
later on. The reader may choose to omit this section in a first
reading and to come back here when necessary.

\subsection{Hyperbolic times}

Here we collect some results from Section~3 and Section~4 in
\cite{OV07} (see also \cite{VV1}), whose proofs we shall omit. Fix
$\vep_0>0$ such that $\sup \phi-\inf \phi < \log\deg(f)-\log q
-\vep_0$ and set $P=\log \la$.

\begin{proposition}\label{p.decreasing.measure}
There exists $\ga_1>0$ such that the measure $\nu(Q_n) \leq
e^{-\ga_1 n}$ for every cylinder $Q_n \in \cQn$. There exists $\ga
\in (0,1)$ and $c_\ga\leq\log q+\vep_0$ such that the cardinality of
cylinders in
$$
B(n,\ga)
    =\big\{Q_n \in \cQ^{(n)} \;|\; \#\{0\leq j \leq n-1 : f^j(Q_n) \subset
                                    \cQ_1 \cup \dots \cup \cQ_q \}>\ga n \big\}
$$
is bounded from above by $\exp(c_\ga n )$ for every large $n$.
Moreover, the measure $\nu(B(\ga,n))$ decreases exponentially fast
as $n\to\infty$.
\end{proposition}

We are now in a position to state the precise condition on the
constant $L>0$ in (H1) that is chosen in a {\bf different way} from
\cite{OV07}. Pick $c>0$ such that
\begin{equation}\label{eq.relation.expansion}
\si^{-(1-\ga)} < e^{-2c}<1
    \quad\text{and}\quad
\log q + c\al+\vep_0 < \log \deg(f),
\end{equation}
where $\al>0$ denotes the H\"older exponent of the potential $\phi$.
Assume that $L$ is sufficiently close to one such that $(\log L)^2
\leq 2 c^2$,
\begin{equation}\label{eq.relation.L}
\si^{-(1-\ga)} L^\ga < e^{-2c}<1
    \quad\text{and}\quad
\sup\phi-\inf\phi<\log\deg(f)-\log q -m\log L.
\end{equation}

\begin{lemma}\label{l.positive.frequency}\cite[Lemma~4.4]{OV07}
There exists $\tau\in(0,1)$ such that, for any $n\ge 1$ and any
$x\notin B(\ga,n)$, there exists $l>\tau n$ and integers $1\le n_1 <
\cdots < n_l$ such that $x$ belongs to the closure of an hyperbolic
cylinder $Q_{n_i}\in\cQ^{n_i}_h$ for every $i=1, \ldots, l$.
Furthermore, $\tau \geq 2c/A$ where $A=\log L$.
\end{lemma}

Since $0<\tau<1$, our choice of $c$ in \eqref{eq.relation.expansion}
guarantees that $\log q + c\tau \al+\vep_0 < \log \deg(f)$. The
following lemma asserts backward distances contraction and a Gibbs
property at hyperbolic times.

\begin{lemma}\label{l.Gibbs.at.hyp.times}
Given $Q_n \in \cQn_h$, $1 \leq j$ and $x,y$ in the closure of
$Q_n$,
$$
d_{f^{n-j}(\bar{Q}^n)}(f^{n-j}(x),f^{n-j}(y))
    \leq e^{-c j} \diam(\cQ).
$$
Moreover, there exists $K>0$ (independent of $n$) such that every $y
\in \ov Q_n$ satisfies
$$
K^{-1} \leq \frac{\nu(Q_n)}{\exp( -P n + S_n \phi(y))} \leq K.
$$
\end{lemma}

As an immediate consequence we obtain that the diameter of most
cylinders decrease exponentially fast. More precisely,

\begin{corollary}\label{c.diameter.vs.positive.frequency}
The diameter of every cylinder $Q_n \notin B(n)$ satisfies
$$
\diam (Q_n) \leq e^{- c \tau n} \diam (\cQ).
$$
\end{corollary}

\begin{proof}
Given $Q_n \not\in B(n)$, there exists a positive integer $k \geq
\tau n$ such that $k$ is a simultaneous hyperbolic time for every
point in $Q_n$. By the mean value theorem
$$
\diam (Q_n)
    \leq e^{-c k} \diam (f^k(Q_n))
    \leq e^{- c \tau n} \diam (\cQ),
$$
which proves the corollary.
\end{proof}

\begin{remark}\label{rem.cardinal.Qn}
Observe that $\# \cQ^{(n)} \leq \#\cQ \, \deg(f)^{n-1}$ for every
positive integer $n$, as an easy consequence of the Markov property
and that every point has $\deg(f)$ preimages. Indeed, given $n \in
\mathbb N$, the Markov assumption on $\cQ$ implies that
$\cQn=f^{-n+1}(\cQ)$. On the other hand, given $Q \in \cQ$,
$f^{-n}(Q)$ is the union of $\deg(f)^n$ cylinders. This shows that
$\# \cQ^{(n)} \leq \#\cQ \, \deg(f)^{n}$ for every $n\ge 1$.
\end{remark}

We say that a measure $\eta$ is \emph{exact} if every element in the
tail sigma-algebra $\cB_\infty = \cap_{j\geq 0} f^{-j}\cB$ is
$\eta$-trivial in the sense that it has measure zero or one.

\begin{lemma}\label{l.exactness}
$\mu$ is exact.
\end{lemma}

\begin{proof}
This is a direct consequence of \cite[Lemma~6.16]{VV1}.
\end{proof}

\subsection{Weak Gibbs property}

Since $\mu$ is absolutely continuous with respect to $\nu$ and the
density $d\mu/d\nu$ is bounded away from zero and infinity (see
\cite{OV07}), then $\mu$ satisfies the non-lacunary Gibbs property.
Here we establish a criterium that relates the decay of the first
hyperbolic time map with a weak Gibbs property similar to the one
introduced by Yuri \cite{Yu99}. Let $n_1$ denote the first
simultaneous hyperbolic time map.

\begin{lemma}\label{l.weak.Gibbs}
There are almost everywhere defined function $(K_n)_{n \ge 1}$ such
that
$$
K_n^{-1}(x)
    \leq \frac{\nu(Q_n(x))}{\exp( -P n + S_n \phi(y))} \leq
K_n(x)
$$
for $\nu$-almost every $x$ and every $y\in \ov Q_n(x)$, and
$$
\mu\Big( x \in M : K_n(x) > a(n) \Big)
        \leq n\; \mu\Big( x \in M : n_1(x)> \frac{\log a(n)}{P+\sup|\phi|} \Big).
$$
In particular, $\lim_{n} \frac1n \log K_n=0$ almost everywhere.
\end{lemma}

\begin{proof}
The proof of this lemma goes along the same ideas in
\cite[Lemma~3.12]{OV07}. Given $n \in \mathbb N$ and $x\in M$ set
$K_n(x)=K\exp[(P+\sup|\phi|)(n_{i+1}(x)-n_i(x))]$, where $i$ is a
positive integer such that $n_i(x) \leq n< n_{i+1}(x)$ and $n_i$
denotes the $i$th simultaneous hyperbolic time map. It is not hard
to show that $K_n$ verifies the Gibbs relation above. Moreover,
\begin{eqnarray*}
\mu\Big(x \in M : K_n(x)> a(n)\Big) &
        \leq \mu \Big( \bigcup_{i} \Big\{x: n_1(f^{n_i(x)}(x))> \frac{\log a(n)}{P+\sup|\phi|} \Big\}
        \Big)\vspace{.2cm}\\
        & \leq n \; \mu\Big( x\in M: n_1(x) > \frac{\log a(n)}{P+\sup|\phi|} \Big),
\end{eqnarray*}
where we made use of the invariance of $\mu$ and that
$n_1(f^{n_i(x)}(x))\ge n_{i+1}(x)-n_i(x)$. The last claim in the
lemma is a direct consequence of the decay estimates.
\end{proof}

\begin{corollary}\label{c.decay.Kn}
There exists $a\in \mathbb N$ such that $\mu$-almost every $x$
satisfies $K_n(x)< n^a$ for all but finitely many values of $n$.
\end{corollary}

\begin{proof}
We use the inclusion $\{n_1> \log a(n)\} \subset B(\ga,\log a(n))$.
If $a\in \mathbb N$ is large enough it follows from the previous
result that
$$
\mu \Big( x \in M : K_n(x) > n^a \Big)
    \leq n \exp\Big (- \frac{c_\ga a}{P+\sup|\phi|} \log n\Big)
    \leq n^{-2},
$$
which is summable. Our claim follows directly from Borel-Cantelli's
lemma.
\end{proof}

This result gives a sufficient condition to obtain sub-exponential
growth of the sequence $(K_n(x))_{n \geq 1}$ in the weak Gibbs
property that will be of particular interest for the proof of the
log-normal fluctuations of the return times in
Section~\ref{sec.fluctuations}.

\subsection{Essential oscillation and variation} \label{subsec.essential}

In this section we present some basic lemmas, needed for the proof
of a Lasota-Yorke inequality in
Subsection~\ref{subsec.spectral.gap}. Given $g \in L^{\infty}(\nu)$
and a set $E$ we define the \emph{essential oscillation} $\ov \osc
(g,E)$ of $g$ on the set $E$ (with respect to $\nu$) as
$$
\ov \osc (g,E)=\nu-ess\sup \{ |g(x)-g(y)|: x,y \in E\}.
$$
Analogously, $\ov\sup(g,E)$ and $\ov\inf(g,E)$ will denote,
respectively, the essential supremum and essential infimum of $g$ in
the set $E$. The following is an immediate consequence of the
triangular inequality.

\begin{lemma}\label{l.oscillation.inequality}
For every $g, h \in L^\infty(\nu)$ and any set $E$ it holds that
$$
\ov\osc(gh,E)
    \leq \ov\osc(g,E) \ov\sup(h, E)+\ov\sup(g, E) \ov \osc(h, E).
$$
\end{lemma}

In the next lemma we give an estimate on the oscillation of the
$\alpha$-H\"older continuous potential $\phi$ in cylinders with
positive frequency of hyperbolic times.

\begin{lemma}\label{l.oscillation.hyp.times}
There exists $C_\phi>0$ such that
$$
\osc (e^\phi, Q_n)
        \leq C_\phi e^{-c \tau \al n} \diam (\cQ)^\al, \;
        \text{for every}\; Q_n \not\in B(n).
$$
\end{lemma}

\begin{proof}
Observe that $e^\phi$ is an $\alpha$-H\"older continuous function
for some positive constant $C_\phi$. Therefore, it follows from
Corollary~\ref{c.diameter.vs.positive.frequency} that
$$
| e^{\phi(x)}- e^{\phi(y)} |
    \leq C_\phi \diam(Q_n)^\al
    \leq C_\phi e^{-c \tau \al n} \diam (\cQ)^\al
$$
for every $Q_n \not\in B(\ga,n)$ and every $x, y \in Q_n$. This
proves the lemma.
\end{proof}

The following lemma plays a key role in the proof of the
Lasota-Yorke inequality.

\begin{lemma}\label{l.sup.osc.L1}
Given a positive $\nu$-measure set $E$ and $g\in L^\infty(\nu)$,
$$
\ov\sup(g,E) \leq \ov\osc(g,E) + \frac{1}{\nu(E)} \int_E |g| \,d\nu.
$$
\end{lemma}

\begin{proof}
Observe that $|g(x)| \leq |g(x)-g(y)| +|g(y)| \leq
\ov\osc(g,E)+|g(y)|$ for almost every $x,y \in E$. In particular,
integrating both sides of the previous inequality with respect to
$y$ it follows that
 $
|g(x)| \leq \ov\osc (g, E) + \frac{1}{\nu(E)} \int_E |g| \,d\nu
 $
for $\nu$-almost every $x\in E$. The lemma is now a direct
consequence of the previous relation.
\end{proof}

Denote by $f^k_{Q_k}$ the restriction of $f^k$ to the cylinder $Q_k$
and observe that it is a bijection onto its image. When no confusion
is possible we will denote by $Q_{n+k}$ the cylinder
$f^{-k}_{Q_k}(Q_n)$.

\begin{lemma}\label{l.sup.concatenation}
Given any positive integers $k$, $n$ and cylinders $Q_k \in \cQk$
and $Q_n \in \cQn$ it holds that
$$
e^{S_{n+k}\phi(f^{-k}_{Q_n}(Q_n))}
    \leq e^{S_n\phi(Q_n)} \, e^{S_k\phi(f^{-k}_{Q_k}(Q_n))}
        \leq e^{(\sup\phi-\inf\phi) k} e^{S_{n+k}\phi(f^{-k}_{Q_k}(Q_n))}.
$$
\end{lemma}

\begin{proof}
Fix $Q_k \in \cQk$ and $Q_n\in\cQn$. The first inequality is
obvious. On the other hand, the Markov property implies that
$f^n(Q_{n+k})= Q_k$. If $x \in Q_n$, $y \in Q_k$ are such that
attain the maximum values in $e^{S_n \phi(Q_n)}$ and $e^{S_k
\phi(Q_k)}$, respectively, then
$$
e^{S_{n+k} \phi(Q_{n+k})} \geq e^{S_{n+k} \phi(f^{-k}_{Q_k}(x))}
                    = e^{S_k \phi( f^{-k}_{Q_k}(x))} \, e^{S_n \phi(x)}.
$$
It follows immediately that
$$
e^{S_k \phi(Q_k)} \, e^{S_n \phi(Q_n)}
    \leq  e^{S_k \phi(y)-S_k \phi( f^{-k}_{Q_k}(x))} \, e^{S_{n+k} \phi(Q_{n+k})}
    \leq e^{(\sup\phi-\inf\phi) k} e^{S_{n+k}\phi(Q_{n+k})},
$$
which proves the lemma.
\end{proof}

Since the diameter of cylinders $Q_n\notin B(n)$ decrease
exponentially fast with $n$, the oscillation of an H\"older
observable over such cylinders is also decreasing.

\begin{lemma}\label{l.oscillation.concatenation}
Given $k \geq 1$ there exists $C_0>0$ (depending on $k$) such that,
if $n$ is large enough then
$$
\osc(e^{S_k \phi}, f^{-k}_{Q_k}(Q_n))
     \leq C_0 \sup(e^{S_k \phi}, f^{-k}_{Q_k}(Q_n)) \;e^{-c \tau \al n}.
$$
for every $Q_k \in \cQ^{(k)}$ and $Q_n \notin B(n)$.
\end{lemma}

\begin{proof}
Let $k \geq 1$ and $Q_k \in \cQ^{(k)}$ be fixed. Since $\phi$ is
$\al$-H\"older continuous there is $C>0$ so that
$$
|S_k \phi (x)-S_k \phi (y)|
    \leq \sum_{j=0}^{k-1} |\phi (f^j(x))-\phi(f^j(y))|
    \leq \sum_{j=0}^{k-1} C \, d(f^j(x),f^j(y))^\alpha
$$
for any $x, y \in f^{-k}_{Q_k}(Q_n)$. The term in the right hand
side is clearly bounded by
$$
C k \max_{0 \leq j \leq k-1} \diam (f^j(Q_{n+k}))^\al.
$$
Recall that $\|Df(x)^{-1}\| \leq L$ for every $x\in M$ by (H1). In
consequence, if $n$ is large enough then
 $
\exp( S_k \phi (x)-S_k \phi (y))
    \leq \exp\Big[ C k \max\{1, L^k\}^{\al} \diam(Q_n)^\al \Big],
 $
which is arbitrarily close to zero by
Corollary~\ref{c.diameter.vs.positive.frequency}. Using that $e^t
\leq 2t$ for every $0<t<1$ we deduce that
 $
|\exp( S_k \phi (x))-\exp( S_k \phi (y))|
    \leq e^{S_k\phi (Q_{n+k})} \, C_0 e^{-c\tau \al n},
 $
where
\begin{equation}\label{eq.C0}
C_0(k)=2C k \max\{1, L^k\}^{\al}
\end{equation}
is independent of $n$. Since $x$ and $y$ where chosen arbitrary this
completes the proof of the lemma.
\end{proof}

\section{Spectral gap for the Ruelle-Perron-Frobenius operator}\label{sec.spectral.gap}

In this section we prove that the Ruelle-Perron-Frobenius operator
has a spectral gap in the space $V_\theta$ of functions of essential
bounded variation for special choices of the parameter $\theta$. As
a consequence we show that the density $d\mu/d\nu$ belongs to
$V_\theta$, and that the equilibrium state $\mu$ has exponential
decay of correlations and satisfies a central limit theorem.

\subsection{Continuity of the Ruelle-Perron-Frobenius operator}

For notational simplicity we denote the Ruelle-Perron-Frobenius
operator $\cL_\phi$ simply by $\cL$. First we show that the operator
$\cL$ is continuous in the Banach space $V_\theta$, provided that
the parameter $\theta$ is small. More precisely,

\begin{lemma}\label{l.continuous.operator}
If $\theta \,e^{\log \deg(f)+ \sup \phi} e^{-c\tau\al}<1$ then $\cL$
is a continuous operator in $V_\theta$: there is a positive constant
$C$ such that
 $
\|\cL g\|_{\theta} \leq C \|g\|_{\theta},
 $
for every $g \in V_\theta$.
\end{lemma}

\begin{proof}
Let $\theta>0$ be such that $\theta \,e^{\log \deg(f)+ \sup \phi}
e^{-c\tau\al} <1$. Given $g \in V_\theta$ we can write
$$
\cL g (x)
    =\sum_{Q \in \cQ} e^{\phi \circ f_Q^{-1}(x)} \, g\circ f_Q^{-1}(x)
    1_{f(Q)}(x).
$$
It is clear that $\|\cL g\|_\infty$ is bounded from above by $\# \cQ
\,e^{\sup \phi} \|g\|_\infty$. Thus, to prove the lemma we are
reduced to show that there exists a constant $C>0$ such that
\begin{equation}\label{eq.continuous.operator}
\var(\cL g)
    :=\sum_{n \geq 0} \theta^n \sum_{Q_n \in \cQn} e^{S_n\phi(Q_n)} \ov\osc (\cL g,Q_n)
    \leq C \|g\|_\theta, \; \forall g \in V_\theta.
\end{equation}
To bound the term involving the oscillation of $\cL g$ notice that,
for every $Q_n \in \cQn$
\begin{multline*}
\ov\osc (\cL g, Q_n)
    \leq \sum_{Q \in \cQ} \ov\osc (e^{\phi \circ f_Q^{-1}} \, g\circ f_Q^{-1} , Q_n) \\
    \leq \sum_{Q \in \cQ}
        \big[ \osc (e^{\phi} , f_Q^{-1}(Q_n)) \;\ov \sup (g) + \sup (e^\phi) \;\ov\osc (g, f_Q^{-1}(Q_n)) \big].
\end{multline*}
Now we deal with the sum over elements $Q_n \in\cQn$ in
\eqref{eq.continuous.operator} by dividing it in two parts,
according to wether $Q_n$ belongs or not to $B(n)$. Since
$\ov\osc(h) \leq 2 \ov\sup (|h|)$ for every $h\in L^\infty(\nu)$ and
$\# B(\ga,n) \leq e^{(\log q + \vep_0)n}$ for every large $n$,
\begin{eqnarray*}
\sum_{\substack{Q_n \in B(n)}} e^{S_n\phi(Q_n)} \ov\osc(\cL g,Q_n)
            & \leq \# B(\ga,n) \, e^{\sup(\phi) n} \times 2 \# \cQ \|e^\phi g \|_\infty \\
            & \leq C_1 e^{(\log q + \sup \phi + \vep_0) n} \|g\|_\infty,
\end{eqnarray*}
for some constant $C_1$ depending only on $\phi$ and $\cQ$. On the
other hand
\begin{multline*}
 \sum_{\substack{Q_n \not\in B(n)}} e^{S_n\phi(Q_n)} \ov\osc(\cL g,Q_n) \\
    \leq \sum_{\substack{Q_n \not\in B(n)}} \sum_{Q \in \cQ}
            e^{S_n\phi(Q_n)} \big[ \osc (e^{\phi} , f_Q^{-1}(Q_n)) \;\ov \sup (g) + \sup (e^\phi) \;\ov\osc (g, f_Q^{-1}(Q_n)) \big].
\end{multline*}
Lemma~\ref{l.oscillation.hyp.times} implies that the right hand side
above is bounded by
\begin{eqnarray*}
C_2 \sum_{\substack{Q_{n+1} \in \cQ^{(n+1)} \\ f(Q_{n+1}) \not\in
        B(n)}} e^{S_{n+1}\phi (Q_{n+1})} C_\phi e^{-c \tau \al n}
        (\diam\cQ)^\al \|g\|_\infty\\
+ C_2\sum_{\substack{Q_{n+1} \in \cQ^{(n+1)} \\ f(Q_{n+1}) \not\in
        B(n)}} e^{S_{n+1}\phi (Q_{n+1})} \ov\osc (g, Q_{n+1})
\end{eqnarray*}
for some positive constant $C_2$ (depending only on $\phi$). We
deduce that there is $C_3>0$ (depending on $\phi$, $\tau$, $\al$,
$\cQ$) such that $\var (\cL g)$ is bounded from above by the sum of
two terms:
\begin{equation*}\tag{a}
\sum_{n=1}^\infty \theta^n \Big[
        C_1 e^{(\log q+\sup\phi+\vep_0)  n}
        + C_3 e^{-c \tau \al n} \sum_{Q_{n+1} \in \cQ^{(n+1)}} e^{S_{n+1}\phi (Q_{n+1})} \Big] \|g\|_\infty
\end{equation*}
and
\begin{equation*}\tag{b}
\frac1\theta \sum_{n=0}^\infty \theta^{n+1}
     C_2\sum_{Q_{n+1} \in \cQ^{(n+1)}} e^{S_{n+1}\phi (Q_{n+1})}\ov\osc
     (g,Q_{n+1}).
\end{equation*}
In particular it follows that $\var (\cL g)$ is bounded by
$$
    C_2 \frac1{\theta}\var(g)
    +\sum_{n=0}^\infty \Big[ C_1\theta^n e^{(\log q + \sup\phi+\vep_0)n}
        + C_3 \#\cQ\theta^n e^{(\log \deg(f) + \sup\phi)n} e^{-c\tau \al n}
        \Big] \|g\|_{\infty}.
$$
Our choice of the parameter $\theta$ together with the relation
$\log q + c\tau\al +\vep_0 < \log \deg(f)$ guarantees the
summability of the previous series and proves that $\cL$ is a
continuous operator in $V_\theta$.
\end{proof}

Let us stress out that the proof of the previous lemma yields the
existence of a constant $C'>0$ such that
 $
\var (\cL g)  \leq C_2\frac1{\theta}\var(g) +C' \|g\|_\infty,
 $
for every small $\theta$ . However, the term $\frac1{\theta}\var(g)$
increases as $\theta$ gets smaller. In particular, the smaller
$\theta$ is the higher the oscillations that may occur in elements
of $V_\theta$.

\subsection{Spectral gap and decay of correlations} \label{subsec.spectral.gap}

Here we prove a Lasota-Yorke inequality for the
Ruelle-Perron-Frobenius operator. This will finally imply on
exponential decay of correlations and central limit theorem for the
equilibrium state. Throughout, let $\theta$ be fixed and such that
\[
\begin{cases}\tag{$\star$}
\theta \, e^{\log \deg(f)+ \inf\phi}>L^\al>1 \\
\theta \, e^{\log \deg(f)+ 2\inf\phi-\sup\phi}>L^\al>1 \\
\theta \, e^{\log \deg(f)+\sup\phi} e^{-c \tau \al}<1
\end{cases}
\]
Some comments on $(\star)$ are in order. Notice that the third
condition is the one required in Lemma~\ref{l.continuous.operator}.
On the other hand, our choice of the parameter $\theta$ impose
certain restrictions on the potential $\phi$. We are now in a
position to state our second small variation condition on $\phi$:
\begin{equation}\label{eq.extra.conditions.potential}\tag{H3b}
\sup\phi-\inf\phi< \frac12 (c\tau-\log L)\al
\end{equation}

\begin{remark}
This condition is clearly satisfied by an open class of nearly
constant potentials. Indeed, by construction, $c\tau \ge 2c^2/\log
L>\log L$. Note also that \eqref{eq.extra.conditions.potential}
legitimates the choice in $(\star)$.
\end{remark}

Let $r_\theta(\cL_\phi)$ and $r(\cL_\phi)$ denote, respectively, the
spectral radius of $\cL_\phi$ in the Banach spaces $V_\theta$ and
$C(M)$. Since $\|\cL^n 1\|_\theta \geq \|\cL^n 1\|_\infty$ for every
$n \in \mathbb N$ then clearly
$$
r_\theta(\cL_\phi)
    = \lim_{n\to\infty} (\|\cL_\phi^n\|_\theta)^\frac1n
    \geq \lim_{n\to\infty} (\|\cL_\phi^n 1\|_\infty)^\frac1n
    \geq \deg(f) e^{\inf\phi},
$$
which proves that $\deg(f)e^{\inf\phi}$ is simultaneously a lower
bound for both spectral radius $r_\theta(\cL_\phi)$ and
$r(\cL_\phi)$. For the time being, let $\la_0$ denote {\bf any}
positive number larger than $\deg(f)e^{\inf\phi}$. Observe that the
previous choice of $\theta$ guarantees that $\theta \la_0>L^\al>1$
and $\theta \la_0 e^{-(\sup\phi-\inf\phi)}>L^\al>1$.

\begin{proposition}\label{p.Lasota.Yorke}
There are $B_1$ and $\xi \in (0,1)$ such that, for every large $k
\geq 1$ there is $B_2(k)>0$ satisfying
$$
\var(\la_0^{-k}\cL^k g)
        \leq B_1 \xi^k \var(g)+B_2(k) \|g\|_{L^1(\nu)}.
$$
\end{proposition}

\begin{proof}
Let $k \geq 1$ be fixed and let $Q_k$ denote the elements of the
partition $\cQk$. Observe that
$$
\cL^k g = \sum_{Q_k} e^{(S_k \phi) \circ f^{-k}_{Q_k}}\;
                    g \circ f^{-k}_{Q_k} \; 1_{f^k(Q_k)}.
$$
Given $g \in V_\theta$, using $\ov\sup(g,Q_k)$ to majorate
$\ov\sup(g,Q_{n+k})$ and Lemma~\ref{l.sup.osc.L1}, it is not hard to
see that $\var(\la_0^{-k}\cL^k g)$ is bounded from above by the sum
of the following three terms:
\begin{equation}\label{ub1}
 (\la_0\theta)^{-k} \sum_{n=0}^\infty \theta^{n+k} \sum_{Q_n} \sum_{Q_k}
            e^{S_n\phi(Q_n)}\osc(e^{S_k\phi},Q_{n+k})\,\ov\osc(g,Q_k),
\end{equation}
\begin{equation}\label{ub2}
 (\la_0\theta)^{-k} \sum_{n=0}^\infty \theta^{n+k} \sum_{Q_n} \sum_{Q_k}
            e^{S_n\phi(Q_n)}\osc(e^{S_k\phi},Q_{n+k})\frac{1}{\nu(Q_k)}\int_{Q_k} g \,d\nu,
\end{equation}
and
\begin{equation}\label{ub3}
 (\la_0\theta)^{-k} \sum_{n=0}^\infty \theta^{n+k} \sum_{Q_n} \sum_{Q_k}
            e^{S_n\phi(Q_n)} e^{S_k\phi(Q_{n+k})} \,\ov\osc(g,Q_{n+k}).
\end{equation}
We deal with these three terms separately. First we treat
\eqref{ub1} by rewriting it as
\begin{equation*}
(\la_0\theta)^{-k} \sum_{n=0}^\infty
            \Big[ \theta^n \sum_{Q_n} e^{S_n\phi(Q_n)} \Big]
            \Big[ \theta^k \sum_{Q_k} \osc(e^{S_k\phi},Q_{n+k})\ov\osc(g,Q_k) \Big]
\end{equation*}
and dividing the sum over elements in $\cQn$ according to wether
they belong or not to $B(n)$. Using $\ov\osc(g,E) \leq 2
\ov\sup(g,E)$ it follows that \eqref{ub1} is bounded by the sum of
the two following terms:
$$
(\la_0\theta)^{-k} \sum_{n=0}^\infty
            \Big[ \theta^n \sum_{Q_n \in B(n)} e^{S_n\phi(Q_n)} \Big]
            \Big[ \theta^k \sum_{Q_k} 2\sup(e^{S_k\phi},Q_k) \, \ov\osc(g,Q_k) \Big]
$$
and
$$
(\la_0\theta)^{-k} \sum_{n=0}^\infty
            \Big[ \theta^n \sum_{Q_n \not\in B(n)} e^{S_n\phi(Q_n)} \Big]
            \Big[ \theta^k \sum_{Q_k} C_0(k)\sup(e^{S_k\phi},Q_k) e^{-c\tau\al n}\, \ov\osc(g,Q_k)\Big],
$$
where $C_0(k)$ is given by \eqref{eq.C0} in
Lemma~\ref{l.oscillation.concatenation}. Our choice on $\theta$
yields that the two previous terms are bounded from above by
$C_0(k)(\la_0\theta)^{-k}\var (g)$ up to finite multiplicative
constants. The constants involved are $2 \sum_{n=0}^\infty (\theta
e^{\log q +\sup\phi+\vep_0})^n$ and $\sum_{n=0}^\infty (\theta
e^{\log \deg(f) +\sup\phi} e^{-c\tau\al})^n$, respectively.

On the one hand, \eqref{ub2} is clearly bounded by
$\|g\|_{L^1(\nu)}$ up to a multiplicative term obtained as the sum
over all $n \geq 0$ of
$$
\la_0^{-k} \max \nu(Q_k)^{-1} \, \theta^n \sum_{Q_n} \sum_{Q_k}
            e^{S_n\phi(Q_n)}\osc(e^{S_k\phi},Q_{n+k}).
$$
Since the measure $\nu$ gives positive weight to any cylinder in
$\cQk$, this shows that there exists a positive constant $K_0(k)$
such that \eqref{ub2} is bounded from above by $\|g\|_{L^1}$ up to
the multiplicative term
$$
K_0(k) \sum_{n=0}^\infty \theta^n \sum_{Q_n} \sum_{Q_k}
            e^{S_n\phi(Q_n)}\osc(e^{S_k\phi},Q_{n+k}).
$$
The part of the sum involving elements $Q_n \in\cQn$ that belong to
$B(n)$ is finite because those elements grow exponentially slow
compared with the allowed size of the oscillations. Indeed,
$$
\sum_{n=0}^\infty \theta^n \sum_{Q_n \in B(n)} \sum_{Q_k}
            e^{S_n\phi(Q_n)}\osc(e^{S_k\phi},Q_{n+k})
    \leq
2\#\cQk e^{k \sup\phi} \sum_{n=0}^\infty \big(\theta e^{\log
q+\sup\phi+\vep_0}\big)^n
$$
is finite. In turn, the sum over elements $Q_n$ that do not belong
$B(n)$ is also finite by Lemma~\ref{l.oscillation.concatenation}:
\begin{align*}
\sum_{n=0}^\infty \theta^n & \sum_{Q_n \not\in B(n)} \sum_{Q_k}
            e^{S_n\phi(Q_n)}\osc(e^{S_k\phi},Q_{n+k})
            \\
 & \leq \sum_{n=0}^\infty \theta^n \sum_{Q_n \not\in B(n)} \sum_{Q_k}
            e^{S_n\phi(Q_n)} C_0 e^{S_k \phi(Q_{n+k})}\,e^{-c \tau \al
            n} \\
 & \leq C_0 \#\cQk e^{k \sup \phi} \sum_{n=0}^\infty \big[\theta
            e^{\log \deg(f)+\sup\phi} \,e^{-c \tau \al}\big]^n <\infty.\\
\end{align*}
This shows that \eqref{ub2} is bounded from above by $\|g\|_{L^1}$
up to a multiplicative constant $B_2(k)$. On the other hand
Lemma~\ref{l.sup.concatenation} ensures that \eqref{ub3} is  bounded
by
\begin{align*}
(\la_0\theta e^{-(\sup\phi-\inf\phi)})^{-k} &
    \sum_{n=0}^\infty \theta^{n+k}\sum_{Q_{n+k} \in \cQ^{(n+k)}}
            e^{S_{n+k}\phi(Q_{n+k})} \,\ov\osc(g,Q_{n+k})\\
    & \leq (\la_0\theta e^{-(\sup\phi-\inf\phi)})^{-k} \var(g).
\end{align*}
It follows that
 $
 \var(\la_0^{-k}\cL^k g) \leq B_1 \xi^k \var(g)+ B_2(k) \|g\|_{L^1},
 $
for a constant $\xi$ is given by
 $
 \xi =\max\{(\la_0\theta)^{-1},(\la_0\theta e^{-(\sup\phi-\inf\phi)})^{-1}\} \frac1k \log C_0(k).
 $
Our first and second conditions on $\theta$ imply that $\xi$ is
strictly smaller than one. This completes the proof of the
proposition.
\end{proof}

As a direct consequence we obtain the following:

\begin{corollary}[Lasota-Yorke Inequality]\label{c.Lasota.Yorke}
There are positive constanst $D_1, D_2$ and $\xi_1 \in (0,1)$ such
that
$$
\var(\la_0^{-n}\cL^n g)
        \leq D_1 \xi_1^n \var(g)+D_2 \|g\|_{L^1(\nu)}.
$$
for every $n \geq 1$.
\end{corollary}

\begin{proof}
Let $B_1$ and $\xi$ be given as in the previous proposition. Pick
$k\geq 1$ such that $\xi_1:=\sqrt[k]{B_1 \xi^k}<1$ and, for any
given $n \ge 1$, write $n=jk+r$ where $j$ is a positive integer and
$0\leq r \leq k-1$. If one applies Proposition~\ref{p.Lasota.Yorke}
recursively and note that $\la_0^{-1}\cL$ does not increase the
$L^(\nu)$-norm, because $\nu$ is conformal, it follows that
$$
\var(\la_0^{-n}\cL^n g)
    \leq \xi_1^{kj} \var(\la_0^{-r}\cL^r g) + B_2(k) \,\Big(\sum_{\ell \ge 0} \xi_1^\ell\Big) \,\|g\|_1.
$$
Moreover, Proposition~\ref{p.Lasota.Yorke} also guarantees that
$$
\var(\la_0^{-r}\cL^n r)
    \leq B_1 \var(g) + \max_{1\le\ell\le k} B_2(\ell) \;\|g\|_1.
$$
The corollary is then immediate taking $D_2=\max_{1\le\ell\le k}
B_2(\ell) +B_2(k) \,\Big(\sum_{\ell \ge 0} \xi_1^\ell\Big)$ and
$D_1=B_1 \xi_1^{-k}$.

\end{proof}

We proceed to prove that the Ruelle-Perron-Frobenius $\cL_\phi$ is
quasi-compact in the functional space $V_\theta$ for any parameter
$\theta$ as above. Since $\la:=r(\cL_\phi)\geq \deg(f) e^{\inf\phi}$
the previous results hold with $\la_0=\la$. First we show that the
iterates of $\la^{-1}\cL_\phi$ are well approximated by operators of
finite rank. Let $A_n$ be the linear operator defined in $V_\theta$
by
$$
A_n(g)= \la^{-n}\cL^n \Big(\mathbb E_\nu(g\mid \cQ^{(n)})\Big)
$$
for every $g \in V_\theta$, where $\mathbb E_\nu( \cdot \mid
\cQ^{(n)})$ stands for the conditional expectation with respect to
the partition $\cQ^{(n)}$. Since the partitions $\cQn$ have finitely
many elements then it is not hard to see that each $A_n$ is a linear
operator of finite rank, hence compact. In addition,

\begin{lemma}\label{l.compact.approximation}
There is $C>0$ and $\xi_1 \in (0,1)$ such that
$$
\|\la^{-n} \cL^n-A_n\|_\theta \leq C \xi_1^n.
$$
\end{lemma}

\begin{proof}
First we bound the $L^\infty(\nu)$ part in $\|\cdot\|_\theta$. Given
$g \in V_\theta$,
\begin{align*}
\|\la^{-n} \cL^n g -  A_n g &\|_\infty
        = \la^{-n} \| \cL^n \big( g -\mathbb E (g \mid \cQn)\big)\|_\infty \\
        & = \la^{-n} \big\| \sum_{Q_n} e^{(S_n\phi)\circ f_{Q_n}^{-n}} \,
            \big[ g\circ f_{Q_n}^{-n}-\mathbb E_\nu(g \mid \cQn)\circ f_{Q_n}^{-n}
            \big] 1_{f^n(Q_n)} \,\big\|_\infty.
\end{align*}
Moreover, for any $Q_n \in \cQn$ the difference between $g$ and
$\mathbb E_\nu(g \mid \cQn)$ computed over the preimages of
$f_{Q_n}^{-n}$ in the term above satisfies
$$
 \big| g\circ f_{Q_n}^{-n}(x)-\mathbb E_\nu(g \mid \!\cQn)\circ f_{Q_n}^{-n}(x)\big|
        \!\!\leq \frac{1}{\nu(Q_n)} \int_{Q_n} | g(f_{Q_n}^{-n}(x))-g(z) | \, d\nu(z)
        \!\!\leq \ov\osc(g,Q_n).
$$
In particular, we deduce that the $L^\infty$ term involved in the
computation of the norm $\|\cdot\|_\theta$ decreases exponentially
fast:
$$
\|\la^{-n} \cL^n g-A_n g\|_\infty
        \leq\la^{-n} \sum_{Q_n} e^{S_n\phi(Q_n)} \, \ov\osc(g,Q_n)
        \leq (\theta\la)^{-n} \var(g).
$$
The variation term in $\|\cdot\|_\theta$ can be bounded analogously,
using Lemma~\ref{l.oscillation.inequality} as in the proof of the
Lasota-Yorke inequality. Indeed,
\begin{align*}
\var(\la^{-n} \cL^n g-A_n g) &
        \leq \la^{-n} \sum_{k=0}^\infty \theta^k \sum_{Q_k} \sum_{Q_n}
        e^{S_k\phi(Q_k)} \ov\osc(e^{S_n\phi},Q_{n+k}) \ov\sup(\ov g_n, Q_{n+k})
        \\ & \\
        & + \la^{-n} \sum_{k=0}^\infty \theta^k \sum_{Q_k} \sum_{Q_n}
        e^{S_k\phi(Q_k)} e^{S_n\phi(Q_{n+k})} \ov\osc(\ov g_n,Q_{n+k}),
\end{align*}
where $\ov g_n = g -\mathbb E(g\,|\,\cQn)$ and
$Q_{n+k}=f^{-n}_{Q_n}(Q_k)$. Since $\mathbb E_\nu(g \,|\,\cQn)$ is
constant over the elements in $\cQn$, clearly $\ov\osc(\ov
g_n,Q_{n+k})=\ov\osc( g,Q_{n+k})$. In consequence the second term in
the right hand side of the sum above coincides with \eqref{ub3},
which in turn is bounded by $(\la \theta
e^{-(\sup\phi-\inf\phi)})^{-n}\var(g)$. On the other hand, the first
term above can be bounded as in \eqref{ub1} by $C(\la \theta
L^{-\al})^{-n} \var(g)$, for some positive constant $C$ that does
not depend on $n$, because
$$
\ov\sup(\ov g_n, Q_{n+k})
        \leq \ov\sup(\ov g_n, Q_{n})
        \leq \ov\osc(\ov g_n, Q_{n}) + \frac{1}{\nu(Q_n)} \int_{Q_n} \ov g_n \,d\nu
$$
and $\int_{Q_n} \ov g_n \, d\nu=0$. In consequence, there exists
$C>0$ and $0<\xi<1$ such that
$$
\var(\la^{-n} \cL^ng-A_n g)
        \leq C \xi^n \|g\|_\theta.
$$
Since this $\theta$-variation term also decreases exponentially fast
as $n$ tends to infinity, this completes the proof of the lemma.
\end{proof}

An interesting consequence of the previous lemma is that the
spectral radius of the Ruelle-Perron-Frobenius operator $\cL_\phi$
in the Banach spaces $C(M)$ and $V_\theta$ do coincide.

\begin{lemma}
$r_\theta(\cL_\phi)=r(\cL_\phi)$.
\end{lemma}

\begin{proof}
The spectral radius $r(\cL)$ of the linear operator $\cL_\phi$ in
the space $C(M)$ of continuous functions is clearly greater or equal
than $\deg(f) e^{\inf\phi}$. Thus, the Lasota-Yorke inequality in
Corollary~\ref{c.Lasota.Yorke} with $\la=r(\cL)$ guarantees that
there exists a uniform constant $C>0$ such that $\var(\la^{-n}
\cL^{n} g) \leq C \|g\|_\theta$ for every $n \in \mathbb N$. Using
$\|\cdot\|_1 \leq \|\cdot\|_\infty$, this proves that there exists a
uniform constant $C>0$ such that $\|\la^{-n} \cL^{n} g\|_\theta \leq
C \|g\|_\theta$ for every $g\in V_\theta$ and $n \in \mathbb N$. In
consequence, the spectral radius $r_\theta(\la^{-1}\cL_\phi)$ of
$\cL_\phi$ in $V_\theta$ verifies
$$
r_\theta(\la^{-1} \cL_\phi) \leq 1.
$$
On the other hand, using once more the conformality of the measure
$\nu$, we get
$$
\|\la^{-n}\cL^n 1\|_\theta
    \geq \|\la^{-n}\cL^n 1\|_\infty
    \geq \|\la^{-n}\cL^n 1\|_{L^1(\nu)}
    =1
$$
for every integer $n\geq 1$, which proves that
$r_\theta(\la^{-1}\cL) \geq 1$. The two estimates above imply
 $
r_\theta(\la^{-1}\cL_\phi)
    =\la^{-1}r_\theta(\cL_\phi)
    =1,
 $
which shows that $r_\theta(\cL_\phi)=\la=r(\cL_\phi)$ and completes
the proof of the lemma.
\end{proof}

We are now in a position to prove the quasi-compactness of the
operator $\la^{-1}\cL$ and, moreover, that it has a spectral gap.

\begin{proposition}
$r_\theta(\la^{-1}\cL_\phi)=1$ and the spectrum
$\si(\la^{-1}\cL_\phi)$ of the operator $\la^{-1}\cL_\phi$ in
$V_\theta$ satisfies
$$
\si(\la^{-1}\cL_\phi)
    \subseteq \big\{ z \in \mathbb C : |z| \leq 1\big\}.
$$
Moreover, $1$ is a simple eigenvalue for $\la^{-1}\cL_\phi$, there
are no other eigenvalues of modulus one and the essential spectral
radius $r_{\text{ess}}(\la^{-1} \cL_\phi)$ is strictly smaller than
one. Furthermore, the density $d\mu/d\nu$ belongs to $V_\theta$.
\end{proposition}

\begin{proof}
Using Naussbaum's formula for the essential spectral radius (see
e.g. ~\cite{DSI} Page 709), which asserts that
$$
r_{\text{ess}}(\la^{-1} \cL_\phi)
    = \lim_{n \to \infty} (\inf \{\|\la^{-n}\cL^n-L\|_\theta : L \,\text{is compact operator}\,\})^\frac1n,
$$
and Lemma~\ref{l.compact.approximation} it follows that
$r_{\text{ess}}(\la^{-1} \cL_\phi) \leq \xi_1$ is strictly smaller
than one. Hence, there is only a finite number of eigenvalues with
finite-dimensional eigenspaces in $\{ z \in \mathbb \si(\la^{-1}
\cL) : |z| > r_{\text{ess}}\}$. Since $r_\theta(\la^{-1}
\cL_\phi)=1$ there must exist some eigenvalue on the unit circle,
and we can write
$$
\la^{-1}\cL_\phi =\Pi_1
        + \sum_{\substack{z\in \si(\la^{-1} \cL_\phi) \\ |z| =1}} z \,\Pi_{z}
        + L_0
$$
where $\Pi_z$ denotes the projection on the subspace associated to
the eigenvalue $z \in \mathbb C$ and $r(L_0)<1$. Using that
$\sum_{j=0}^{n-1} z^j$ is uniformly bounded in norm for every $n$ it
follows that
$$
\Big\| \frac1n \sum_{j=0}^{n-1} \la^{-j}\cL_\phi^j - \Pi_1
\Big\|_\theta \xrightarrow[n \to \infty]{} 0.
$$
On the other hand, using $\|\cdot\|_\theta \geq \|\cdot\|_\infty
\geq \|\cdot\|_{L^1(\nu)}$ one gets
$$
 \big\| \frac1n \sum_{j=0}^{n-1} \la^{-j}\cL_\phi^j 1\big\|_\theta
    \geq \big\| \frac1n \sum_{j=0}^{n-1} \la^{-j}\cL_\phi^j 1 \big\|_{L^1(\nu)}
    =1, \quad\text{for every $n \in \mathbb N$.}
$$
This shows that $\Pi_1$ is nonzero and that $h=\Pi_1(1)$ is an
eigenfunction for $\la^{-1}\cL_\phi$ associated to the eigenvalue
$1$. Up to a normalization in $L^1(\nu)$ it is not difficult to see
that $\hat \mu=h\nu$ is an $f$-invariant probability measure: for
every $g \in C(M)$
$$
\int g\circ f \,d\hat \mu
    = \int \la^{-1} \cL_\phi( g\circ f h) \,d\nu
    = \int \la^{-1} \cL_\phi(h) g \,d\nu
    = \int g \,d\hat \mu.
$$
Since $\hat\mu$ is absolutely continuous with respect to $\nu$ then
it is an equilibrium state. By uniqueness of the equilibrium state,
$\hat\mu$ must coincide with $\mu$ and, in particular, $d\mu/d\nu=h$
belongs to $V_\theta$. In fact the same argument yields that $1$ is
a simple eigenvalue, thus, the only eigenvalue of modulus one.

It remains only to rule out the existence of other eigenvalues of
modulus one distinct from $1$. Let $z \in \mathbb C$ and $h' \in
V_\theta$ be such that $\la^{-1}\cL_\phi h'=z h'$ and $|z|=1$. Since
$h$ is bounded from below by some constant $C_1>0$ the function
$\psi$ defined by $h'=h \psi$ belongs to $L^2(\mu)$ because
 $
\|\psi\|_{L^2(\mu)}
        \leq \|h'\|_\infty^2 / {C_1^2}
        <\infty.
 $
The Koopman operator $U:L^2(\mu) \to L^2(\mu)$ acts on each $g \in
L^2(\mu)$ by $U(g)=g \circ f$. By construction $\psi$ is an
eigenfunction for the dual operator $U^*$. Indeed, $U^*(\psi)=z
\psi$ because
$$
\int (U^*\psi) g \,d\mu
    =\! \int \psi (g\circ f) \,d\mu
    =\! \int h'(g\circ f) \,d\nu
    =\! \int \la^{-1} \cL_\phi (h'(g\circ f)) \,d\nu
    =\! \int (z h') g \,d\nu.
$$
In consequence, $\psi=z^n U^n(\psi)=z^n \,(\psi \circ f^n)$ is
measurable with respect to the sigma-algebra $f^{-n}(\cB)$ for every
$n \geq 1$. Since $\mu$ is exact (recall Lemma~\ref{l.exactness})
then $\psi$ must be constant, which proves that $h'$ belongs to the
subspace generated by $h$ and consequently $z=1$. This shows that
$1$ is the only eigenvalue of modulus one and completes the proof of
the proposition.
\end{proof}

\begin{corollary}\label{c.correlation.decay}
There are $C>0$, $\xi \in (0,1)$ such that every $\Phi \in L^1(\nu)$
and $\Psi \in V_\theta$ satisfy $ C_n(\Phi,\Psi) \leq C \xi^n
\|\Psi\|_\theta \|\Phi\|_{L^1(\nu)}$. Moreover,
$$
\Big|\mu(Q_k \cap f^{-n}(Q_l)) - \mu(Q_k) \mu(Q_l)\Big|
            \leq C \xi^n \mu(Q_l)
$$
for every $n\ge 1$ and every pair of cylinders $Q_l \in \cQ^{(l)}$
and $Q_k \in \cQk$.
\end{corollary}

\begin{proof}
Given $\Phi \in L^1(\nu)$ and $\Psi \in V_\theta$,
$$
\int \Phi (\Psi \circ f^n) d\mu
    = \int \Phi h (\Psi \circ f^n)  d\nu
    = \int (\la^{-n}\cL^n)(\Phi h) \, \Psi \,d\nu
$$
for every $n \geq 1$. Hence,
\begin{multline*}
 \Big| \int \Phi (\Psi \circ f^n) d\mu - \int \Phi d\mu \int \Psi d\mu \Big|
        = \Big| \int \Big[ (\la^{-n}\cL^n)(\Phi h)- h \int (\Phi h) \,d\nu \Big]
        \, \Psi \,d\nu \Big| \\
        \leq
        \Big\| (\la^{-n}\cL^n)(\Phi h)- h \int (\Phi h) \,d\nu
        \Big\|_\theta \|\Psi\|_{L^1(\nu)}.
\end{multline*}
Since $h \int (\Phi h) \,d\nu$ is the projection of $\Phi h$ in the
one-dimensional eigenspace associated to the eigenvalue $1$ it is a
consequence of the spectral gap that the previous term is bounded by
$C \xi^n \|\Psi\|_\theta \|\Phi\|_{L^1(\nu)}$. On the other hand,
the second claim is an immediate consequence of the first one
provided that we show that the characteristic function $1_{Q}$ of a
cylinder $Q \in \cQ^{(k)}$ belongs to $V_\theta$. Since
$\ov\osc(\cdot) \leq 2 \ov\sup(\cdot)$, it is clear that
$$
\var(1_{Q})
    = \sum_{n \geq 0} \theta^n \sum_{Q_n} e^{S_n\phi(Q_n)}
    \ov\osc(1_{Q}, Q_n) \leq 2 \sum_{n \geq 0} (\theta
    e^{\sup\phi})^n
$$
is finite. In consequence $\|1_{Q}\|_\theta= 1+\var(1_{Q})$ is also
finite, which proves our claim and finishes the proof of the
corollary.
\end{proof}

Finally, to complete the proof of Corollary~\ref{c.CLT} it is enough
to prove the following:

\begin{lemma}\label{l.Vtheta.contains.Holder}
$V_\theta$ contains the space of $\alpha$-H\"older observables.
\end{lemma}

\begin{proof}
Let $g$ be an $\alpha$-H\"older continuous observable. Since
$\|g\|_\infty$ is finite it remains to estimate $\var (g)$. Indeed,
dividing the sum of the elements in $\cQn$ according to wether they
belong to $B(n)$ or not, it is not hard to check that there is
$C_g>0$ such that
$$
\var(g)
    \leq  2 \|g\|_\infty \sum_{n \geq 0} \theta^n \# B(n) e^{\sup\phi n}
    + C_g \sum_{n \geq 0} \theta^n e^{(\log\deg(f)+\sup\phi-c\tau\al)n}.
$$
This proves the lemma.
\end{proof}

\section{Exponential distribution of hitting
times}\label{sec.exponentialdistribution}

In this section we combine ideas from \cite{GaSc97} and \cite{Pac00}
with the weak Gibbs property and the strong mixing properties of the
equilibrium state $\mu$ to study the hitting times asymptotics in
Theorem~\ref{thm.first.entrance.distribution}. For notational
simplicity, given $Q\in\cQn$ we set
$$
g_Q(t)=\mu\Big(\tau_Q > \frac{t}{\mu(Q)} \Big).
$$
The strategy to prove Theorem~\ref{thm.first.entrance.distribution}
is to consider a large subset $\cQn_\vep$ of $n$-cylinders such that
$g_Q(\sqrt{\mu(Q)})$ behaves essentially as $\exp(-\mu(Q))$ for
every $Q \in \cQn_\vep$ and to explore the strong mixing property of
$\mu$ to obtain many instants of independence. We will need some
preliminary results.

\begin{lemma}\cite[Lemma 2]{GaSc97}\label{l.Pac1}
For every measurable set $E$ and all positive $t$,
$$
\mu\Big(\tau_E \leq \frac{t}{\mu(E)}\Big) \leq t + \mu(E).
$$
\end{lemma}

If $\gamma_1>0$ is given by Proposition~\ref{p.decreasing.measure}
then we have the following:

\begin{lemma}\label{l.special.cylinders}
There is $K>0$ such that for any $\vep>0$ there is a subset
$\cQn_\vep$ of $n$-cylinders of measure at least $1-\vep$ and
satisfying
\begin{equation}\label{eQ}
e^{-\sqrt{\mu(Q)} (1 + K e^{-\ga_1 n/2})}
    \leq
g_Q(\sqrt{\mu(Q)})
    \leq
e^{-\sqrt{\mu(Q)} (1 - K e^{-\ga_1 n/2})}
\end{equation}
for every $Q \in \cQn_\vep$ and every large $n$.
\end{lemma}

\begin{proof}
First observe that if $n$ is large enough and $Q\in \cQn$ arbitrary
\begin{align*}
-\log g_Q(\sqrt{\mu(Q)})
        & = - \log \Big[1- \mu\Big(\tau_Q\leq \frac{\sqrt{\mu(Q)}}{\mu(Q)}\Big)\Big] \\
        & \leq \mu\Big(\tau_Q\leq\frac{\sqrt{\mu(Q)}}{\mu(Q)}\Big) +
            \Big[\mu\Big(\tau_Q\leq
            \frac{\sqrt{\mu(Q)}}{\mu(Q)}\Big)\Big]^2.
\end{align*}
Then Proposition~\ref{p.decreasing.measure} and Lemma~\ref{l.Pac1}
imply that the later sum is bounded from above by $\sqrt{\mu(Q)}(1+K
e^{-\ga_1 n/2})$ for some positive constant $K$, which proves the
lower bound in \eqref{eQ}.

In the other direction, let $\cQn_\vep$ be the family of
$n$-cylinders that have no self returns in the time interval
$[1,\zeta n]$ for some $\zeta>0$: the $n$-cylinder $Q$ belongs to
$\cQn_\vep$ if $f^j(Q)$ does not intersect $Q$ for every $1 \leq j
\leq \zeta n$. Any element in $\cQn \setminus \cQn_\vep$ has short
recurrence and is completely characterized by $\zeta n$ cylinders of
the partition $\cQ$. Consequently, $\# [\cQn \setminus \cQn_\vep
]\leq (\#\cQ)^{\zeta n}$ for every $n$. In particular, if
$\zeta=\zeta(\vep)$ is chosen small enough then
$$
\mu\Big( \cup\{Q \in \cQn : Q \not\in \cQn_\vep\} \Big)
        \leq (\# \cQ)^{\zeta n} e^{-\ga_1 n}<\vep
$$
for every large $n$. We claim that every $Q\in \cQn_\vep$ verifies
the upper bound in \eqref{eQ}. On the one hand
$$
-\log g_Q(\sqrt{\mu(Q)})
        \geq 1 - g_Q(\sqrt{\mu(Q)})
        = \mu\Big(\tau_Q\leq \frac{\sqrt{\mu(Q)}}{\mu(Q)}\Big).
$$
Following \cite[Lemma~3]{GaSc97} and \cite[Lemma~5.3]{Pac00}, one
gets
$$
\mu\Big(\tau_Q \leq \frac{t}{\mu(Q)}\Big)
    \geq \frac{t^2}{t^2+\mu(Q)(1+t)+t(1+Ke^{-\ga_1 n})}
$$
for every $Q \in \cQn_\vep$. Together with the previous inequality
this shows that
$$
-\log g_Q(\sqrt{\mu(Q)})
        \geq \frac{\mu(Q)}{\mu(Q)[2+\sqrt{\mu(Q)}]+\sqrt{\mu(Q)}[1+Ke^{-\ga_1 n}]}
$$
which is larger than $\sqrt{\mu(Q)}(1-K e^{-\ga_1 n/2})$. This
completes the proof of the lemma.
\end{proof}

The strong mixing property of $\mu$ guarantees some independence of
the system.

\begin{lemma}\label{l.independence}
There exists $C>$0 such that if $n \geq 1$ is sufficiently large
then
$$
\sup_{s \geq \sqrt{\mu(Q)}}
    \Big| g_Q(\sqrt{\mu(Q)}+s)- g_Q(\sqrt{\mu(Q)}) g_Q(s) \Big|
    \leq C \mu(Q)^{3/4}
$$
for every cylinder $Q \in \cQn$.
\end{lemma}

\begin{proof}
This proof is similar to the one of \cite[Lemma~5.4]{Pac00}, which
explores the strong mixing properties of the system. We include a
brief sketch of the proof for completeness reasons.

Let $Q$ be any fixed cylinder of $\cQn$. Given positive $t,s$ and a
small $\Delta$ (to be chosen later on), by invariance of $\mu$ it
follows that $|g_Q(t+s)-g_Q(t)\,g_Q(s)|$ is bounded from above by
the sum of the following three terms:
\begin{equation}\label{eq.1}
 \Bigg|g_Q(t+s)- \mu\Big( \tau_Q \notin [0,\frac{t}{\mu(Q)}] \cup [\frac{t}{\mu(Q)}+\Delta , \frac{t+s}{\mu(Q)}] \Big)\Bigg|,
\end{equation}
\begin{equation}\label{eq.2}
 \Bigg|\mu\Big( \tau_Q \notin [0,\frac{t}{\mu(Q)}] \cup [\frac{t}{\mu(Q)}+\Delta , \frac{t+s}{\mu(Q)}] \Big)- g_Q(t) \; \mu\Big(\tau_Q \notin [\Delta , \frac{s}{\mu(Q)}] \Big)\Bigg|,
\end{equation}
and
\begin{equation}\label{eq.3}
g_Q(t) \, \Bigg| \mu\Big(\tau_Q \notin [\Delta , \frac{s}{\mu(Q)}]
\Big) - \mu\Big(\tau_Q \notin [0 , \frac{s}{\mu(Q)}] \Big)\Bigg|.
\end{equation}
Since \eqref{eq.1} is the measure of the set of points that do enter
$Q$ in the time interval $[\frac{t}{\mu(Q)},
\frac{t}{\mu(Q)}+\Delta]$ then it is bounded by $\Delta \mu(Q)$.
Similarly, \eqref{eq.3} is also bounded by $\Delta \mu(Q)$. For the
remaining term, computations analogous to the ones in
\cite[page~356]{Pac00} guarantee that
$$
\eqref{eq.2}
    = |\int g_1 \, (g_2\circ f^{\De+1}) \,d\mu -\int g_1 \,d\mu \; \int g_2 \,d\mu|
    \leq C \xi^{\Delta+1} \|g_1\|_\theta \|g_2\|_1,
$$
where
$$
g_1=1_{Q^c}\,\frac1h (\la^{-1}\cL_{Q^c})^{[\frac{t}{\mu(Q)}]}(h)
 \quad \text{and}\quad
g_2=\prod_{j=0}^{[\frac{s}{\mu(Q)}]-\De-1} 1_{Q^c} \circ f^j,
$$
and $\cL_Q(g)$ stands for $\cL(g 1_Q)$. Clearly $\frac1h\in
V_\theta$, because $h \in V_\theta$ and $h$ is bounded away from
zero. Observe that $\|g_2\|_1\le 1$ and that $\|\la^{-1}\cL
h\|_\infty =\|h\|_\infty$ is finite. Hence, using that $\var(f_1
f_2) \le \var(f_1) \|f_2\|_\infty+\var(f_2) \|f_1\|_\infty$ then
\begin{align*}
 \eqref{eq.2}  & \leq  C \xi^{\De+1} \Big\| 1_{Q^c}\,\frac1h (\la^{-1}\cL_{Q^c})^{[\frac{t}{\mu(Q)}]}(h
                         )\Big\|_\theta\\
                & \leq  C \xi^{\De+1} \Bigg[
                        \var (\frac1h) \Big\|1_{Q^c}\,(\la^{-1}\cL_{Q^c})^{[\frac{t}{\mu(Q)}]}(h
                                )\Big\|_\infty \\
                & \qquad\qquad + \Big\|\frac1h \Big\|_\infty \var\Big(1_{Q^c}\,
                                (\la^{-1}\cL_{Q^c})^{[\frac{t}{\mu(Q)}]}(h
                        )\Big) \Bigg] \\
                & \leq C' \xi^{\De+1} \Big[1 + \var\Big((\la^{-1}\cL_{Q^c})^{[\frac{t}{\mu(Q)}]}(h
                        )\Big)\Big]
\end{align*}
for some positive constant $C'$. To estimate the term in the right
hand side above we use (7) in page 358 of \cite{Pac00}: for every
$N\ge 1$
$$
\cL^N_{Q^c}(h)= h - \sum_{r=0}^{N-1} \cL^r \cL_{Q}(h)
            + \sum_{0\le i+j\le N-2} \cL^i \cL_{B_{i,j}}(h),
$$
where $B_{i,j}= Q \cap f^{-1}(Q^c) \cap \dots \cap f^{-N+i+j+2}(Q^c)
\cap f^{-N+i+j+1}(Q)$ is a cylinder of order $n+N-i-j-1$  contained
in $Q$. So, using the Lasota-Yorke inequality it is not hard to
obtain
$$
\var(\la^{-N}\cL^N_{Q^c}(h)) \le C'' ( N+ N^2) \,\var(h)
$$
for some positive constant $C''$, and shows that
$$
|g_Q(t+s)-g_Q(t)\,g_Q(s)|
    \le 2 \Delta \mu(Q) + C \xi^{\Delta+1} (1+2C''[\frac{t}{\mu(Q)}]^2)
    \le C \mu(Q)^{\frac34}
$$
for some $C>0$ provided that $t=\sqrt{\mu(Q)}$, $s\ge \sqrt{\mu(Q)}$
and $\Delta=\mu(Q)^{-\frac14}$. This completes the proof of the
lemma.
\end{proof}

We finish this section with the following:

\begin{proof}[Proof of Theorem~\ref{thm.first.entrance.distribution}]
Let $t \geq 0$, $n \geq 1$ and $\vep>0$ be fixed, and let
$\cQn_\vep$ be as in the previous lemma. Take $Q \in \cQn_\vep$ and
set $k=[\frac{t}{\sqrt{\mu(Q)}}]$. The strategy is to divide the
estimate on the distribution of the entrance time $g_Q(t)$ in blocks
where some independence holds. Write $t=k \sqrt{\mu(Q)}+r$, with
$0\leq r < \sqrt{\mu(Q)}$, and note that
\begin{multline*}\tag{$\dagger$}
\Big| g_Q(t)- e^{-t}\Big|
    \leq \Big| g_Q(t)- g_Q(k \sqrt{\mu(Q)}) \Big|
    + \Big| g_Q(k \sqrt{\mu(Q)}) - g_Q(\sqrt{\mu(Q)} )^k \Big| \\
    + \Big| g_Q(\sqrt{\mu(Q)})^k - e^{-k \sqrt{\mu(Q)}}\Big|
    + \Big| e^{-k \sqrt{\mu(Q)}}- e^{-t} \Big|.
\end{multline*}
The last term in the right hand side above is clearly bounded from
above by $2 \sqrt{\mu(Q)}$, while we can use the invariance of $\mu$
to get
$$
\Big| g_Q(t)- g_Q(k \sqrt{\mu(Q)}) \Big|
    = \mu\Big( \frac{k \sqrt{\mu(Q)}}{\mu(Q)}<\tau_Q \leq
    \frac{t}{\mu(Q)} \Big)
    \leq \mu\Big( 0 <\tau_Q \leq \frac{\sqrt{\mu(Q)}}{\mu(Q)} \Big),
$$
which is bounded by $\sqrt{\mu(Q)}+\mu(Q)$ according to
Lemma~\ref{l.Pac1}. Since $Q$ belongs to $\cQn_\vep$ then \eqref{eQ}
holds and the third term in $(\dagger)$ decays exponentially fast
with $n$ because it is bounded from above by
\begin{multline*}
2k \Big[ -\log g_Q(\sqrt{\mu(Q)}) - \sqrt{\mu(Q)}\Big] \;
    \Big[ g_Q(\sqrt{\mu(Q)})^k + e^{-k \sqrt{\mu(Q)}} \Big] \\
    \leq 2k \sqrt{\mu(Q)} K e^{-\ga_1 n/2} \; \Big[ 2 e^{-k \sqrt{\mu(Q)} (1-Ke^{-\ga_1 n/2})} \Big].
\end{multline*}
Finally, we deal with the second term in the right hand side of
$(\dagger)$. Indeed, it is not difficult (see e.g.
\cite[Lemma~6]{GaSc97}) to use induction in
Lemma~\ref{l.independence} and obtain
\begin{equation}\label{eq.Galves.Schmitt}
\Big| g_Q(k \sqrt{\mu(Q)})- g_Q(\sqrt{\mu(Q)} )^k \Big|
    \leq C \frac{\mu(Q)^{3/4}}{1-g_Q(\sqrt{\mu(Q)})}.
\end{equation}
Using the last inequalities in the proof of
Lemma~\ref{l.special.cylinders}, every $Q\in\cQn_\vep$ satisfies
$1-g_Q(\sqrt{\mu(Q)}) \geq \sqrt{\mu(Q)} (1-K e^{-\ga_1 n})$, which
guarantees that the second term in $(\dagger)$ satisfies
$$
\Big| g_Q(k \sqrt{\mu(Q)})- g_Q(\sqrt{\mu(Q)} )^k \Big|
        \leq \frac{\mu(Q)^{3/4}}{\sqrt{\mu(Q)} (1-K e^{-\ga_1 n})}
        \leq \mu(Q)^{1/4}
$$
and decreases exponentially fast with $n$. This completes the proof
of the theorem.
\end{proof}

\section{Fluctuations of the return times}\label{sec.fluctuations}

This section is devoted to the proof of
Theorem~\ref{thm.recurrence.fluctuations}. We explore the
exponential asymptotic distribution of hitting times, the weak Gibbs
property of $\mu$ and the Central Limit Theorem to obtain the
log-normal distribution of the return times. The following relation
between hitting times and return times, similar to Lemma~4.1 in
\cite{Pac00}, is a consequence of the good mixing properties for
$\mu$.

\begin{lemma}\label{l.hitting.vs.recurrence}
Let $(t_n)$ be a sequence such that $\lim_{n\to\infty} t_n/n \to
+\infty$. Then
$$
\lim_{n\to\infty}
    \Big| \mu (R_n>t_n) - \sum_{Q \in \cQn} \mu(Q) \;\mu\Big(\tau_Q > t_n \Big) \Big|
    =0.
$$
\end{lemma}

\begin{proof}
Since
$$
\mu(R_n > t)
        = \sum_{Q \in \cQn} \mu(Q \cap \{\tau_Q >  t\})
$$
we will estimate the differences $\mu(Q \cap \{\tau_Q>t\})-\mu(Q)
\mu(\tau_Q>t)$ for elements $Q\in\cQn$. In fact, given $k<n<r<t$ and
$Q\in \cQn$ write
\begin{align}
(\ast) = \mu(Q \cap  & \{\tau_Q>t\})-\mu(Q) \mu(\tau_Q>t) \nonumber \\
 &\leq  \mu(Q \cap \{\tau_Q>t\})-\mu(Q \cap \tau_Q\notin[r,t])\label{eq.indep1}\\
 &+   \mu(Q \cap \tau_Q\notin[r,t])- \mu(Q) \mu(\tau_Q \notin[r,t]) \label{eq.indep2}\\
 &+   \mu(Q) \Big[\mu(\tau_Q\notin[r,t])-\mu(\tau_Q>t)\Big] \label{eq.indep3}
\end{align}
It is clear that \eqref{eq.indep3} coincides with $\mu(Q)
\mu(\tau_Q<r)$, which is bounded from above by $r \mu(Q)^2$, because
$\mu(\tau_Q<r)\leq \mu(\cup_{j\leq r} f^{-j} Q)$. Moreover, by
exponential decay of correlations
$$
|\eqref{eq.indep2}|
    \leq |\mu( Q \cap f^{-r}(\cap_{j=0}^{t-r} f^{-j}(Q^c)) -\mu(Q) \mu(f^{-r}(\cap_{j=0}^{t-r}f^{-j}(Q^c)))|
    \leq K \xi^r.
$$
On the other hand, $|\eqref{eq.indep1}|$ is bounded from above by
$\sum_{j=0}^r \mu(Q \cap f^{-j}(Q))$. To deal with this last term we
will consider differently wether $Q$ belongs to the set $E_{<k}$ of
$n$-cylinders such that there exists $0\leq i\leq k$ so that
$f^{-i}(Q) \cap Q \neq \emptyset$, or not. In fact, summing up over
elements in $\cQn$ and using the exponential decay of correlations
it follows that
\begin{align*}
\sum_{Q\in\cQn} |\eqref{eq.indep1}|
    \leq \sum_{Q\in E_{<k}} & \sum_{j=0}^r \mu(Q \cap f^{-j}(Q))
        + \sum_{Q\in E_{<k}^c} \sum_{j=0}^r \mu(Q \cap f^{-j}(Q))\\
    &\leq r \mu (E_{<k}) +\sum_{Q\in E_{<k}^c} \sum_{j=k}^r [K' \xi^j +\mu(Q)] \;\mu(Q)\\
    &\leq r \; \mu (E_{<k}) +   r [K' \xi^k +e^{-\ga_1 n}] \\
\end{align*}
where $K'$ is a constant that involves the constant $K$ from decay
of correlations and an upper bound for $\|1_Q\|_\theta$. Using that
$\mu (E_{<k}) \leq  \#\cQk e^{-\ga_1 n}$ and the previous estimates
we obtain
$$
\sum_{Q\in\cQn} |(\ast)|
    \leq r e^{-\ga_1 n}
    + K \# \cQn \xi^r
    + r \; \#\cQk e^{-\ga_1 n} +  r [K' \xi^k +e^{-\ga_1 n}].
$$
The previous inequality holds for $r(n)=\min\{t_n, n^2\}$ and
$k(n)=\Big[\frac{\ga_1}{2\log\deg(f)} n \Big]$. In fact, we get
\begin{align*}
\Big| \mu & (R_n>t_n)  - \sum_{Q \in \cQn} \mu(Q) \;\mu\Big(\tau_Q > t_n \Big) \Big| \\
    & \leq n^2 e^{-\ga_1 n}
        + K \# \cQ \deg(f)^n \xi^{r(n)}
        + n^2 \; \#\cQ e^{-\ga_1 n/2} +  n^2 [K' \xi^{k(n)} +e^{-\ga_1 n}]. \\
\end{align*}
The expression in the right hand side above tends to zero as
$n\to\infty$. Indeed, note that the second term in the right hand
side above tends to zero because $r(n)/n \to \infty$, by
construction. This completes the proof of the lemma.
\end{proof}
\vspace{.2cm}

\begin{proof}[Proof of Theorem~\ref{thm.recurrence.fluctuations}]
Let $\phi$ be an H\"older continuous potential as above such that
$\si=\si(\phi)>0$, and fix $\vep>0$ arbitrary small. Given $t\in
\mathbb R$ and $n \ge 1$
\begin{align*}
\mu\Big(\frac{\log R_n-n h_\mu(f)}{\si \sqrt n}>t\Big)
        & =\mu(R_n > e^{n h_\mu(f)} e^{\si t \sqrt n}) \\
        & = \sum_{Q \in \cQn} \mu(Q \cap \{\tau_Q >  e^{n h_\mu(f)} e^{\si t \sqrt n}\}).
\end{align*}
Let $\cQn_\vep$ be the family of cylinders given by
Theorem~\ref{thm.first.entrance.distribution}. Since $\cup\{Q \in
\cQn_\vep\}$ has $\mu$-measure at least $1-\vep$ and the first
entrance time $\tau_Q$ of every cylinder $Q\in\cQn_\vep$ has
exponential distribution up to a small error, then
$$
\mu\Big(\frac{\log R_n-n h_\mu(f)}{\si \sqrt n}>t\Big)
        = \! \sum_{Q \in \cQn_\vep} \mu(Q) \;\mu\Big(\tau_Q > e^{n h_\mu(f)} e^{\si t \sqrt n}\Big)
        + \cO(e^{-\beta n})+ \cO(\vep).
$$
This is a consequence of the previous
Lemma~\ref{l.hitting.vs.recurrence} with $t_n=e^{nh_\mu(f)+\si
t\sqrt n}$. Since $\vep>0$ was chosen arbitrary, to show the
log-normal distribution of the return times we are left to prove the
convergence
\begin{equation}\label{eq.final.fluctuation}
 \lim_{n \to \infty} \Bigg[ \sum_{Q \in \cQn_\vep} \mu(Q) \; e^{-\mu(Q)e^{n h_\mu(f)} e^{\si t \sqrt n}}\Bigg]
    = \frac{1}{\sqrt{2\pi}}\int_{t}^\infty e^{-\frac{x^2}{2}} \,dx + \cO(\vep).
\end{equation}
By Lemma~\ref{l.weak.Gibbs} and Corollary~\ref{c.decay.Kn}, there is
$a\in \mathbb N$ and for $\mu$-almost every $x$ there is a sequence
$(K_n)_n$ (depending on $x$) satisfying $K_n(x)\leq n^a$ for all but
finitely many values of $n$ and such that
$$
K_n(x)^{-1}
        \leq
\frac{\mu(Q_n(x))}{e^{-Pn + S_n \phi(x)}}
        \leq
K_n(x)
$$
for every $n \geq 1$. Using also $\mu(\cup \{Q : Q \notin
\cQn_\vep\})<\vep$ and $h_\mu(f)+ \int \phi \,d\mu=\Ptop(f,\phi)=P$,
for any $\rho>0$
\begin{align*}
\sum_{Q \in \cQ^{(n)}_\vep} & \mu(Q) \; e^{-\mu(Q)e^{n h_\mu(f)}
e^{\si t \sqrt n}}
        = \int_{\cup \{Q : Q \in \cQn_\vep\}} e^{-\mu(Q(x)) e^{n h_\mu(f)} e^{\si t \sqrt n}}
        d\mu(x) \\
    & \geq e^{-e^{-\rho \si\sqrt n}}
    \Big[\mu\Big( x \in M \mid e^{-\mu(Q_n(x))e^{n h_\mu(f)} e^{\si t \sqrt n}} > e^{-e^{-\rho \si\sqrt n}}\Big) -\vep\Big]\\
    & \geq e^{-e^{-\rho \si\sqrt n}}
    \Big[\mu\Big( x \in M \mid \frac{-S_n\phi(x)+ n \int \phi d\mu}{\si \sqrt n}  > t + \rho  + \frac{1}{\sqrt n}\log
    K_n(x) \Big)-\vep\Big].
\end{align*}
Since $\phi$ belongs to $V_\theta$ (recall
Lemma~\ref{l.Vtheta.contains.Holder}) and it satisfies the Central
Limit Theorem (see Corollary~\ref{c.CLT}), taking the limit as
$n\to\infty$ and $\rho\to 0$ we obtain that
$$
 \liminf_{n \to \infty} \Bigg[ \sum_{Q \in \cQn_\vep} \mu(Q) \; e^{-\mu(Q)e^{n h_\mu(f)} e^{\si t \sqrt n}}\Bigg]
    \geq \frac{1}{\sqrt{2\pi}}\int_{t}^\infty e^{-\frac{x^2}{2}} \,dx -\vep.
$$
The upper estimate in \eqref{eq.final.fluctuation} is obtained
analogously. Indeed, for any $\rho>0$
\begin{align*}
\sum_{Q \in \cQ^{(n)}_\vep} & \mu(Q) \; e^{-\mu(Q)e^{n h_\mu(f)}
e^{\si t \sqrt n}}
        = \int_{\cup \{Q : Q \in \cQn_\vep\}} e^{-\mu(Q(x)) e^{n h_\mu(f)} e^{\si t \sqrt n}} d\mu(x) \\
 & \leq \mu\Big( x \in \cup \{Q : Q \in \cQn_\vep\} \mid e^{-\mu(Q_n(x))e^{n h_\mu(f)} e^{\si t \sqrt n}} > e^{-e^{-\rho \si\sqrt n}}\Big)\\
 & \leq e^{e^{-\rho \si\sqrt n}}
    \Big[\mu\Big( x \in M \mid e^{-\mu(Q_n(x))e^{n h_\mu(f)} e^{\si t \sqrt n}} > e^{-e^{-\rho \si\sqrt n}}\Big) +\vep\Big]\\
 & \leq e^{e^{-\rho \si\sqrt n}}
    \Big[\mu\Big( x \in M \mid \frac{-S_n\phi(x)+ n \int \phi d\mu}{\si \sqrt n}  > t + \rho  - \frac{1}{\sqrt n}\log
    K_n(x) \Big)+\vep\Big].
\end{align*}
taking the limit as $n\to\infty$ and $\rho\to 0$ one gets
$$
 \limsup_{n \to \infty} \Bigg[ \sum_{Q \in \cQn_\vep} \mu(Q) \; e^{-\mu(Q)e^{n h_\mu(f)} e^{\si t \sqrt n}}\Bigg]
    \leq \frac{1}{\sqrt{2\pi}}\int_{t}^\infty e^{-\frac{x^2}{2}} \,dx +\vep,
$$
which proves the upper bound in \eqref{eq.final.fluctuation}. The
proof of the theorem is now complete.
\end{proof}


\medskip \textbf{Acknowledgements:} Most of this work was done at IMPA,
during the author's PhD under the guidance of Marcelo Viana. The
author greatly acknowledges IMPA excellent conditions and scientific
environment. The author is grateful to  V. Ara\'ujo for reading a
preliminary version of this paper and making important suggestions,
and to A. Galves for providing a copy of \cite{GaSc97}. This work
was partially supported by Funda\c c\~ao para a Ci\^encia e
Tecnologia (FCT-Portugal) by the grant SFRH/BD/11424/2002) and by
Funda\c c\~ao de Amparo \`a Pesquisa do Estado do Rio de Janeiro -
FAPERJ.


\bibliographystyle{alpha}

\end{document}